%% file: main.tex
\documentclass[reqno,12pt]{amsart}
\usepackage{amssymb}
\usepackage{mathrsfs}
\usepackage{txfonts}
\usepackage{amsfonts}
\usepackage{amstext}
\usepackage{amssymb}
\usepackage{amsmath}
\usepackage{graphicx}
\usepackage{hyperref}
\usepackage{url}
\usepackage{amssymb}
\usepackage{subfig}

\hypersetup{
        colorlinks   = true,
        citecolor    = blue,
        linkcolor    = blue,
        urlcolor     = blue
}
\allowdisplaybreaks
\newtheorem{theorem}{Theorem}[section]

\newtheorem{lemma}[theorem]{Lemma}
\newtheorem{definition}[theorem]{Definition}

\renewcommand{\theequation}{\thesection.\arabic{equation}}

\newenvironment{acknowledgement}{\smallskip{\sc Acknowledgement.}\rm}{\smallskip}

\numberwithin{equation}{section}
\newcounter{counterConstant}

\input{tcilatex}
\input{tcirm}

\begin{document}
\title[]{Weak monotonicity property of Korevaar-Schoen norms on nested
fractals}
\author[Chang]{Diwen Chang}
\address{Department of Mathematical Sciences, Tsinghua University, Beijing
100084, China.}
\email{cdw20@mails.tsinghua.edu.cn}
\author[Gao]{Jin Gao}
\address{School of Mathematics, Hangzhou Normal University, Hangzhou 310036,
China.}
\email{gaojin@hznu.edu.cn}
\author[Yu]{Zhenyu Yu}
\address{Department of Mathematics, College of Science, National University
of Defense Technology, Changsha 410073, China.}
\email{yuzy23@nudt.edu.cn}
\author[Zhang]{Junda Zhang}
\address{School of Mathematics, South China University of Technology,
Guangzhou 510641, China.}
\email{summerfish@scut.edu.cn}
\date{}

\begin{abstract}
In this paper, we study the weak monotonicity property of $p$-energy related
Korevaar-Schoen norms on connected nested fractals for $1<p<\infty$. Such
property has many important applications on fractals and other metric
measure spaces, such as constructing $p$-energies (when $p=2$ this is
basically a Dirichlet form), generalizing the classical Sobolev type
inequalities and the celebrated Bourgain-Brezis-Mironescu convergence.
\end{abstract}

\maketitle

%\tableofcontents

\section{Introduction}

Let $(M,d,\mu )$ be a metric measure space with $0<\text{diam}(M)<\infty $
(where `diam' denotes the diameter) throughout the paper. For $p>1$, $\sigma
>0$, we say that weak monotonicity property (NE) holds with $\sigma >0$ on $%
(M,d,\mu )$, if there exists $C>0$ such that for all $u\in KS_{p,\infty
}^{\sigma }$ (see Definition \ref{dfNE}), 
\begin{align*}
& \sup_{r\in (0,\text{\textrm{diam}}(M))}\frac{1}{r^{p\sigma }}\int_{M}\frac{%
1}{\mu (B(x,r))}\int_{B(x,r)}|u(x)-u(y)|^{p}d\mu (y)d\mu (x) \\
& \leq C\liminf_{r\rightarrow 0}\frac{1}{r^{p\sigma }}\int_{M}\frac{1}{\mu
(B(x,r))}\int_{B(x,r)}|u(x)-u(y)|^{p}d\mu (y)d\mu (x),
\end{align*}%
where the metric ball $B(x,r):=\{y\in M:d(x,y)<r\}$.

For $p=2$, property (NE) can be regarded as a form of the monotonicity of a
related family of Dirichlet forms, and is not hard to verify on certain
metric measure spaces with a suitable exponent $\sigma $. It is used to
prove the $\Gamma $-convergence of non-local Dirichlet form to local
Dirichlet form on p.c.f. self-similar sets in \cite{GuLau.2020.AASFM} and on
the Sierpi\'{n}ski carpet \cite{Grigoryanyang.2019.trans}. By the equivalent
Besov semi-norms of Dirichlet forms, the celebrated
Bourgain-Brezis-Mironescu (BBM) convergence (see \cite%
{BourgainBrezisMironescu.2001.439}) of Besov semi-norm is also obtained by
(NE) . In the nonlinear setting $p\neq 2$, (NE) is essentially used in \cite%
{GaoYuZhang2022PA} and \cite{GaoYuZhang2023} to obtain the generalisation of
the BBM convergence for $p$-energies on metric measure spaces, which is a
nonlinear generalisation of Dirichlet forms since the bi-linearity is
missing.

We briefly introduce the $p$-energies defined on fractals and general metric
measure spaces, which attracts considerable attention recently (see \cite%
{BaudoinLecture2022,Caoqiugu2022adv,GaoYuZhang2022PA,GaoYuZhang2023,Kigami2022penergy,Shimuzu.2022}%
). For a smooth Euclidean area $\Omega $, the $p$-energy is simply defined
as $\int_{\Omega }|\nabla u(x)|^{p}dx$, but for fractals or general metric
measure spaces, it is not easy to define proper gradient structures in \cite%
{StrichartzWong04}. One kind of the construction of $p$-energy ($1<p<\infty $%
) is based on the graph-approximation to the spaces, see \cite[Section 6]%
{Caoqiugu2022adv} for p.c.f. fractals and \cite%
{Kigami2022penergy,Shimuzu.2022} for other metric spaces. Another kind uses
the Korevaar-Schoen (or Besov) norms, for example in \cite{GaoYuZhang2022PA}%
, the generalised BBM convergence is used to construct $p$-energy on
homogeneous p.c.f. self-similar sets. The property (NE) can be regarded as
certain weak monotonicity of Korevaar-Schoen (or Besov) norm type. Similar
properties in terms of heat kernel-based type (KE) and discretized type (VE)
were considered in \cite{GaoYuZhang2023}, where (KE) is a crucial assumption
in the main theorem of \cite{AlonsoBaudoinchen2020JFA}. 
%Besides, byMoreover, this weak monotonicity property is used to $p$-energy form in\cite{GaoYuZhang2022PA}, which is also used to obtain the convergence of $p$-energy forms and Bosev semi-norms in p.c.f. self-similar set.

A pretty natural and important question is to check the validity of (NE) on
typical fractal spaces for all $p>1$ (with suitable exponents) instead of
the special case $p=2$. In \cite{BaudoinLecture2022}, Baudoin showed that
the weak monotonicity property (NE) can be deduced from the combination of
Poinc\'{a}re inequality and capacity conditions in measure metric spaces,
and then verified (NE) on the Vicsek set and the Sierpi\'{n}ski gasket for
all $p>1$. Moreover, Baudoin left an open question that the validity of (NE)
for some other spaces in \cite{BaudoinLecture2022}. (NE) is very recently
verified on the Sierpi\'{n}ski carpet by Yang \cite{Yang23} and
Murugan-Shimizu \cite{MuruganShimizu23} with two different arguments (for
all $p>1$). It is also very recently proved that (NE) holds for nested
fractals in \cite{GaoYuZhang2023}, where the authors mainly use the
equivalence of different forms of weak monotonicity properties (termed (VE)
therein). In this paper, we use a different argument based on Baudoin's
method in \cite{BaudoinLecture2022} to verify (NE) for nested fractals.

We will introduce necessary definitions to formally state the main theorem
in Section \ref{secdf}, and present our proof in Section \ref{sec3}.

\textbf{Notation}: The letters $C$,$C^{\prime }$,$C^{\prime\prime }$,$C_{i}$
are universal positive constants depending only on $M$ which may vary at
each occurrence.

\section{Definitions and the Main Result}

\label{secdf} In this section, we introduce the definitions of
Korevaar-Schoen spaces and nested fractals, and then build up the main
theorem.

\subsection{The Korevaar-Schoen space}

Let $p>1$, $\sigma >0$. For $u\in L^{p}(M;\mu )$, $r>0$, denote 
\begin{equation}
\Phi _{u}^{\sigma }(r):=\frac{1}{r^{p\sigma }}\int_{M}\frac{1}{\mu (B(x,r))}%
\int_{B(x,r)}|u(x)-u(y)|^{p}d\mu (y)d\mu (x).  \label{phi_u}
\end{equation}%
The Korevaar-Schoen space $KS_{p,\infty }^{\sigma }$ in \cite[Section 4.2]%
{BaudoinLecture2022} is defined as 
\begin{equation}
KS_{p,\infty }^{\sigma }:=\{u\in L^{p}(M,\mu )|\limsup_{r\rightarrow 0}\Phi
_{u}^{\sigma }(r)<\infty \}.  \label{ks_energy}
\end{equation}

%
%
%
%\begin{equation}
%\Vert u\Vert^p _{KS_{p,\infty }^{\sigma }}=\limsup_{r\rightarrow 0}\int_M%
%\frac{1}{V(x,r)}\int_{B(x,r )}\frac{|u(x)-u(y)|^{p}}{r ^{p\sigma }}d\mu
%(y)d\mu (x)<\infty.  \label{ks_energy1}
%\end{equation}
We define semi-norm for $u\in L^{p}(M,\mu )$ by 
\begin{equation*}
\Vert u\Vert _{KS_{p,\infty }^{\sigma }}^{p}:=\limsup_{r\rightarrow 0}\Phi
_{u}^{\sigma }(r).
\end{equation*}

We restate the definition of property (NE) formally.

\begin{definition}
\label{dfNE}We say that a metric measure space $(M,d,\mu )$ satisfies
property (NE) with $\sigma >0$, if there exists $C>0$ such that, for all $%
u\in KS_{p,\infty }^{\sigma }$, 
\begin{equation}
\sup_{r\in (0,\mathrm{diam}(M))}\Phi _{u}^{\sigma }(r)\leq
C\liminf_{r\rightarrow 0}\Phi _{u}^{\sigma }(r).  \label{NE1}
\end{equation}
\end{definition}

\subsection{Nested fractals}

In this section, we introduce a class of highly symmetric p.c.f. fractals,
namely, the nested fractals introduced by Lindstr{\o }m in \cite%
{Lindstrom.1990.}. Typical examples of nested fractals are the Sierpi\'{n}%
ski gasket and Vicsek set see Figure \ref{fig1}. We will follow the notation
from Kumagai in \cite{Kumagai.1993.PTaRF205}.

\begin{figure}[tbph]
\centering
\subfloat[ Sierpi\'{n}ski gasket]{\includegraphics[width=6cm]{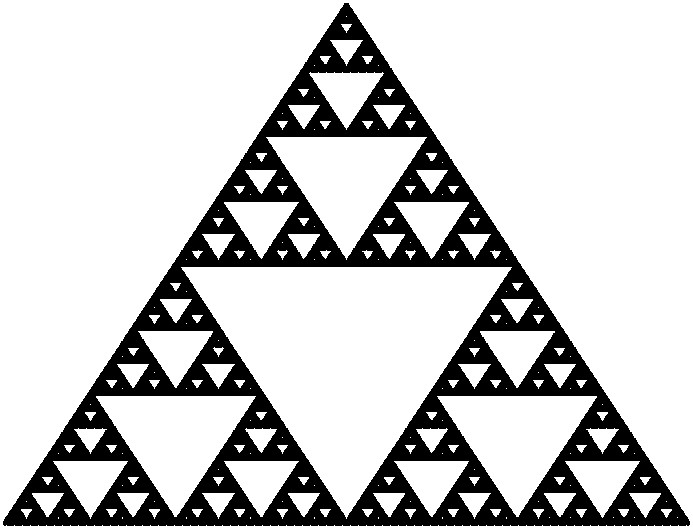}}%
\quad \quad\quad\quad 
\subfloat[ Vicsek
set]{\includegraphics[width=6cm]{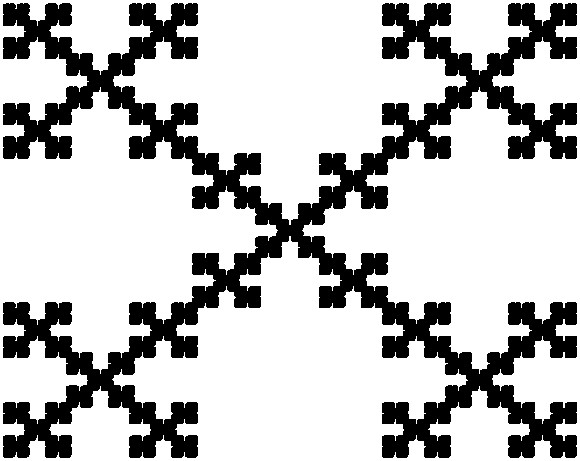}}
\caption{Two typical examples of nested fractals}
\label{fig1}
\end{figure}

We focus on homogeneous self-similar p.c.f. IFSs $\left\{ F_{i}\right\}
_{i=1}^{N}$ in $\mathbb{R}^{D}$. The word `\emph{homogeneous}' means that
the mappings are in the form 
\begin{equation}
F_{i}(x)=\rho \left( x-b_{i}\right) +b_{i},\quad 1\leq i\leq N  \label{f_i}
\end{equation}%
where $\rho \in (0,1)$ is the contraction ratio. Next, we introduce the
natural symbolic space associated with the IFS. Assume that $K$ is the
attractor of the IFS $\left\{ F_{i}\right\} _{i=1}^{N}$. Let $W=\{1,2,...,N\}
$, $W_{n}$ be the set of words with length $n$ over $W$, and $W^{\mathbb{N}}$
be the set of all infinite words over $W$. Let $w=\mathrm{w_{1}w_{2}}...\in
W^{\mathbb{N}}$, the canonical projection $\pi :$ $W^{\mathbb{N}}\rightarrow
K$ is defined by $\{\pi (w)\}:=\bigcap\limits_{n\in \mathbb{N}^{\ast }}F_{%
\mathrm{w_{1}...w_{n}}}(K), $ where $F_{\mathrm{w_{1}...w_{n}}}:=F_{\mathrm{%
w_{1}}}\circ ...\circ F_{\mathrm{w_{n}}}$. Then we can define the critical
set $\Gamma $ and the post-critical set $\mathcal{P}$ by 
\begin{equation*}
\Gamma =\pi ^{-1}\left( \bigcup_{1\leq i<j\leq N}\left( K_{i}\cap
K_{j}\right) \right) ,\quad \mathcal{P}=\bigcup_{m\geq 1}\tau ^{m}(\Gamma ),
\end{equation*}%
where $K_{i}=F_{i}(K),\tau $ is the left shift by one index on $W^{\mathbb{N}%
}$ (see \cite[Definition 1.3.13]{Kigami.2001.}). We say the IFS is\emph{\
post-critically finite} (p.c.f.) if $\mathcal{P}$ is finite. For any $w\in
W_{n}$, define 
\begin{equation*}
V_{0}=\pi (\mathcal{P}),\quad V_{n}=\bigcup_{w\in W_{n}}F_{w}\left(
V_{0}\right) ,\quad V_{\ast }=\bigcup_{n\geq 1}V_{n},
\end{equation*}%
then $K$ is the closure of $V_{\ast }$ and $V_{w}:=F_{w}\left( V_{0}\right) $%
. For $w\in W_n$, $K_{w}:=F_{w}\left( K\right) $ is called a $n$-cell of $K$.

Now we give the definition of nested fractals. Let $F$ be the set of all
fixed points of $\{F_{i}\}_{i=1}^{N}$. We say $x\in F$ is an essential fixed
point, if there exists $y\in F,y\neq x$ and $i,j\in \{1,2,\cdots ,N\}$ so
that $F_{i}(x)=F_{j}(y)$. Let $F^{(0)}$ be the set of all essential fixed
points.

\begin{definition}[Nested Fractals]
We say $K$ is a nested fractal, if the four conditions below are all
satisfied:

i) The IFS $\{F_i\}_{i=1}^N$ is homogeneous and satisfies the open set
condition;

ii) Connectivity: for any pair of $1$-cells $K_{i},K_{j}$, there exists a
chain $\{{a_{k}\}}_{k=0}^{n}\subset \{1,2,\cdots ,N\}$ so that $%
a_{0}=i,a_{n}=j$ and $K_{a_{k}}\cap K_{a_{k+1}}\neq \emptyset $ for any $%
0\leq k<n$;

iii) Symmetry: For any $x,y\in \mathbb{R}^{D}$, denote 
\begin{equation*}
H_{xy}:=\{z\in \mathbb{R}^{D}:d(z,x)=d(z,y)\},
\end{equation*}%
and let $U_{xy}$ be the reflection related to $H_{xy}$. For any $n\geq 1$
and for any $x_{0}\neq y_{0}\in F^{(0)}$, $U_{x_{0}y_{0}}$ maps $n$-cells to 
$n$-cells, and maps any $n$-cell containing points in both sides of $%
H_{x_{0}y_{0}}$ to itself.

iv) Nesting: For any $n\geq 1$, and any $w,v\in W^n$ with $w\neq v$, then 
\begin{equation}  \label{nest}
K_w\cap K_v=F_w(F^{(0)})\cap F_v(F^{(0)}).
\end{equation}
\end{definition}

According to the above definition, it is easy to see that nested fractals
are homogeneous p.c.f. self-similar fractals. Hence we will directly apply
related results on homogeneous p.c.f. self-similar set to nested fractals in
Section \ref{sec3}.

Without loss of generality, we can assume that \textrm{diam}$(K)$=1 by an
affine transformation on $\{b_{i}\}_{i=1}^{n}$. From now on, whenever we
mention a nested fractal, we denote it by $K$ and assume it is connected
with \textrm{diam}$(K)$=1. Since the open set condition (OSC) is fulfilled,
the $\alpha $-dimensional Hausdorff measure of $K$, denoted by $\mu $,
exists and is \emph{$\alpha $-regular}, namely, the measure of any metric
ball $B(x,r)$ with $0<r\leq 1$ in $K$ satisfies 
\begin{equation}
C^{-1}r^{\alpha }\leq \mu (B(x,r))\leq Cr^{\alpha },  \label{alpha_r}
\end{equation}%
where $C\geq 1$ is a constant and $\alpha =\text{dim}_{H}(K)=-\log N/\log
\rho $.

\subsection{The main result}

From now on, we assume that $M$ is a nested fractal $K$, with $d$ being the
Euclidean metric restricted to $M$ and $\mu $ being the $\text{dim}_{H}(M)$%
-dimensional Hausdorff measure. Define semi-norm for $u\in L^{p}(K,\mu )$ by 
\begin{equation*}
\lbrack u]_{B_{p,\infty }^{\sigma }}^{p}=\sup_{r\in (0,1)}\Phi _{u}^{\sigma
}(r),
\end{equation*}%
with its domain 
\begin{equation*}
B_{p,\infty }^{\sigma }:=B_{p,\infty }^{\sigma }(K):=\{u\in L^{p}(K,\mu )|\
[u]_{B_{p,\infty }^{\sigma }}<\infty \}.
\end{equation*}%
Clearly, the Besov space $B_{p,\infty }^{\sigma }$ coincides with the
Korevaar-Schoen space $KS_{p,\infty }^{\sigma }$.

We define the \emph{$p$-critical exponent of $(K,d,\mu )$} by%
\begin{equation*}
\sigma _{p}^{\#}=\sup \{\sigma >0:B_{p,\infty }^{\sigma }\text{ contains
non-constant functions}\}.
\end{equation*}%
This critical exponent is important, because it is proved by \cite[Lemma 4.7]%
{BaudoinLecture2022} that, property (NE) can only hold with $%
\sigma\geq\sigma _{p}^{\#},$ but the space $B_{p,\infty }^{\sigma }$ is
trivial when $\sigma>\sigma _{p}^{\#}.$ So from now on, whenever we mention
property (NE), we assume taking $\sigma=\sigma _{p}^{\#}$ and omit it (since
the underlying space is fixed).

\begin{theorem}
\label{thm1} Property (NE) holds in connected nested fractals.
\end{theorem}

\section{Proof of the main theorem}

\label{sec3} In this section, we present the proof of our main Theorem \ref%
{thm1}. Before that, we collect some necessary technical results from \cite%
{Caoqiugu2022adv,GaoYuZhang2022PA,GaoYuZhang2023}.

\subsection{Property (E)}

Let $K$ be a homogeneous p.c.f. self-similar set. Define 
\begin{equation*}
E_{n}^{(p)}(u):=\sum_{x,y\in V_{w},|w|=n}|u(x)-u(y)|^{p},\ \mathcal{E}%
_{n}^{\sigma }(u):=\rho ^{-n(p\sigma -\alpha )}E_{n}^{(p)}(u).
\end{equation*}

Denote a discretization of the $B_{p,\infty }^{\sigma }$-norm by 
\begin{equation*}
\mathcal{E}_{p,\infty }^{\sigma }(u)=\sup\limits_{n\geq 0}\mathcal{E}%
_{n}^{\sigma }(u).
\end{equation*}

We introduce the weak monotonicity property (E) for $K$, following the
definition of $p$-energy form in \cite{GYZ2022F}.

\begin{definition}
(\cite[Definition 3.1]{GaoYuZhang2022PA}) \label{dfE} We say that a
homogeneous p.c.f. self-similar set $K$ satisfies \textrm{property (E) with
respect to} $\sigma >\alpha /p$ if there exists a positive constant $C$ such
that

\begin{itemize}
\item[(i)~] for any $u\in B_{p,\infty }^{\sigma }$ and for all $n\geq 1$, $%
\mathcal{E}_{0}^{\sigma }(u)\leq C\mathcal{E}_{n}^{\sigma }(u)$,

\item[(ii)] for any $u\in l(V_{0})$, there exists an extension $\tilde{u}\in
B_{p,\infty }^{\sigma }$.
\end{itemize}
\end{definition}

Note that property (E) (i) implies the following weaker property: 
\begin{equation*}
\sup_{n\geq 0}\mathcal{E}_{n}^{\sigma}(u)\leq C\liminf_{n\rightarrow \infty }%
\mathcal{E}_{n}^{\sigma}(u).
\end{equation*}

\begin{lemma}
(\cite[Lemma 4.10]{GaoYuZhang2023}) \label{lemma:E}For a nested fractal,
property (E) holds with $\sigma=\sigma_{p}^\#$ and $\sigma
_{p}^{\#}>\alpha/p $ for all $1<p<\infty $. \label{yyqx}
\end{lemma}

Unfortunately, not all homogeneous p.c.f. fractals satisfy property (E),
since property (E) guarantees that $\sigma_p^\ast=\sigma_p^\#$ (see \cite[%
Proposition 3.4]{GaoYuZhang2022PA}), where 
\begin{equation*}
\sigma _{p}^{\ast }=\sup \{\sigma >0:B_{p,\infty }^{\sigma }\cap C(K)\text{
is dense in }C(K)\}.
\end{equation*}
Readers may see \cite{GuLau.2020.TAMS} for counter-examples when $p=2$.

\subsection{$p$-harmonic extension}

In this subsection, we will construct an auxiliary function whose energy can
be controlled by its discretization, and the process is based on the
existence of fixed point in \cite{Caoqiugu2022adv}.

Actually, for any nested fractal $K$, according to the proof of \cite[%
Proposition 4.10]{GaoYuZhang2023}, a byproduct will be obtained, which
indicates the equivalence of two $p$-energies defined in \cite%
{Caoqiugu2022adv} and \cite{GaoYuZhang2022PA}. We will use partial results
from \cite{Caoqiugu2022adv} and \cite{GaoYuZhang2022PA} as follows. Let $%
s\in (0,1)$ and define an operator $\Lambda $ maps $p$-energies on $l(V_{0})$
to $p$-energies on $l(V_{1})$, such that for any $u\in l(V_{1})$, 
\begin{equation}
\Lambda E(u):=\frac{1}{s}\sum_{i=1}^{N}E(u\circ F_{i}).  \label{E-1}
\end{equation}%
By \cite[Theorem 4.2]{Caoqiugu2022adv} (see also \cite[equation (4.34)]%
{GaoYuZhang2023}), there exists a fix point $E$ defined on $l(V_{0})$ of the
renormalization map $\mathcal{T}$ such that 
\begin{equation}
\mathcal{T}E=E,  \label{TE}
\end{equation}%
where the operator $\mathcal{T}$ is defined as in \cite[Definition 3.1]%
{Caoqiugu2022adv}. For this $E$, by the proof of \cite[Lemma 3.2(b)]%
{Caoqiugu2022adv}, we know that for all $m\in \mathbb{N}^{+}$ 
\begin{equation*}
\lbrack \Lambda ^{m}E]_{V_{m-1}}=\Lambda ^{m-1}\mathcal{T}E=\Lambda ^{m-1}E,
\end{equation*}%
that is 
\begin{equation}
\Lambda ^{m-1}E(u)=\min \{\Lambda ^{m}E(v):\ v\in l(V_{m}),\ v\big|%
_{V_{m-1}}=u\}.  \label{AM}
\end{equation}%
Therefore, for any $u\in l(V_{n})$, we can define its $p$-harmonic
extensions $\{u_{m}\}_{m=n+1}^{\infty }$ satisfying 
\begin{equation*}
\Lambda ^{n+1}E(u_{n+1})=\Lambda ^{n+1}E(u)\ \ with\ \ u_{n+1}\in
l(V_{n+1}),\ u_{n+1}\big|_{V_{n}}=u,
\end{equation*}%
and for $m\geq n+2$, 
\begin{equation*}
\Lambda ^{m}E(u_{m})=\Lambda ^{m-1}E(u_{m-1})\ \ with\ \ u_{m}\in l(V_{m}),\
u_{m}\big|_{V_{m-1}}=u_{m-1},
\end{equation*}%
then we can define $H_{p}(u)\in l(V_{\ast })$ satisfying 
\begin{equation}
H_{p}(u)=\lim_{m\rightarrow \infty }u_{m}.  \label{HP}
\end{equation}%
Besides, we have by \cite[equations (4.37), (4.38), (4.39)]{GaoYuZhang2023}
that 
\begin{equation}
\sigma _{p}^{\#}=\dfrac{\log _{\rho }s+\alpha }{p},  \label{impo}
\end{equation}%
and for any positive integer $n$, 
\begin{equation}
\mathcal{E}_{n}^{\sigma _{p}^{\#}}(u)\simeq \Lambda ^{n}E(u)  \label{En}
\end{equation}%
for all $u\in B_{p,\infty }^{\sigma _{p}^{\#}}$.

\begin{lemma}
\label{Hp2} Let $K$ be a nested fractal and $n\in \mathbb{N}^{+}$. Then
there exists a constant $C>0$ such that for all $u\in l(V_{n})$, 
\begin{equation}
\mathcal{E}_{p,\infty }^{\sigma _{p}^{\#}}(H_{p}(u))\leq C \mathcal{E}%
_{n}^{\sigma _{p}^{\#}}(u),  \label{En-1}
\end{equation}%
where $H_{p}(u)$ is defined as in (\ref{HP}).
\end{lemma}

\begin{proof}
Note that for $k>n$, by (\ref{En}) and (\ref{HP}), 
\begin{equation}
\mathcal{E}_{k}^{\sigma _{p}^{\#}}(H_{p}(u))\simeq \Lambda
^{k}E(H_{p}(u))=\Lambda ^{k}E(H_{p}(u)\big|_{V_{k}})=\Lambda ^{n}E(u)\simeq 
\mathcal{E}_{n}^{\sigma _{p}^{\#}}(u).  \label{Hp-1}
\end{equation}
Also, by (\ref{AM}) and (\ref{En}), we have for any $1\leq j\leq n$,%
\begin{equation}
\Lambda ^{n}E(u)=\Lambda ^{n}E(H_{p}(u)\big|_{V_{n}})\geq \Lambda
^{j}E(H_{p}(u)\big|_{V_{j}})\simeq \mathcal{E}_{j}^{\sigma
_{p}^{\#}}(H_{p}(u)).  \label{Hp-2}
\end{equation}

Combining (\ref{Hp-1}) and (\ref{Hp-2}), 
\begin{equation}
\mathcal{E}_{p,\infty }^{\sigma _{p}^{\#}}(H_{p}(u))=\sup_{k\geq 1}\mathcal{E%
}_{k}^{\sigma _{p}^{\#}}(H_{p}(u))\leq C\mathcal{E}_{n}^{\sigma
_{p}^{\#}}(u).  \label{En-2}
\end{equation}
which completes the proof.
\end{proof}

Next, we prove that $H_{p}(u)$ (defined as in (\ref{HP})) has a unique H\"{o}%
lder continuous extension to $K$, and we still denote it by $H_{p}(u)$.

\begin{lemma}
\label{Hp} Let $K$ be a nested fractal and $n\in \mathbb{N}^{+}$. Then there
exist two positive constants $c,C$ such that, for all $u\in l(V_{n})$ and $%
a,b\in V_{\ast }$ with $d(a,b)<c\rho$, 
\begin{equation}
|H_{p}(u)(a)-H_{p}(u)(b)|^{p}\leq Cd(a,b)^{p\sigma _{p}^{\#}-\alpha } 
\mathcal{E}_{n}^{\sigma _{p}^{\#}}(u),  \label{HHP}
\end{equation}
where $H_{p}(u)$ is defined as in (\ref{HP}).
\end{lemma}

\begin{proof}
For $a,b\in V_{\ast }$ with $d(a,b)<c\rho $ (where the constant $c$ will be
determined later), we consider the following three cases:

Case $1$: $a,b\in V_{w}$ for some $|w|=n$, we have 
\begin{eqnarray*}
\mathcal{E}_{p,\infty }^{\sigma _{p}^{\#}}(H_{p}(u)) &=&\sup\limits_{n\geq
0}\rho ^{-n(p\sigma _{p}^{\#}-\alpha )}E_{n}^{(p)}(H_{p}(u)) \\
&=&\sup\limits_{n\geq 0}\rho ^{-n(p\sigma _{p}^{\#}-\alpha )}\sum_{x,y\in
V_{w},|w|=n}|H_{p}(u)(x)-H_{p}(u)(y)|^{p} \\
&\geq &\sup\limits_{n\geq 0}\rho ^{-n(p\sigma _{p}^{\#}-\alpha
)}\sup\limits_{x,y\in V_{w},|w|=n}|H_{p}(u)(x)-H_{p}(u)(y)|^{p} \\
&\geq &C_1\sup\limits_{n\geq 0}d(x,y)^{-(p\sigma _{p}^{\#}-\alpha
)}\sup\limits_{x,y\in V_{w},|w|=n}|H_{p}(u)(x)-H_{p}(u)(y)|^{p},
\end{eqnarray*}%
which implies%
\begin{equation}
\sup\limits_{x,y\in V_{w},|w|=n}|H_{p}(u)(x)-H_{p}(u)(y)|^{p}\leq
C_1^{-1}d(x,y)^{p\sigma _{p}^{\#}-\alpha }\mathcal{E}_{p,\infty }^{\sigma
_{p}^{\#}}(H_{p}(u))\leq C_2 d(x,y)^{p\sigma _{p}^{\#}-\alpha }\mathcal{E}%
_{n}^{\sigma _{p}^{\#}}(u),  \label{Hp-3}
\end{equation}%
where we have used (\ref{En-1}) in the last inequality. Then (\ref{HHP})
holds by (\ref{Hp-3}).

Case $2$: There exists $|w|=n$ such that $b\in V_{w}$ for some $n\in \mathbb{%
N}^{+}$ and $a\in K_{w}\cap V_{m}$ for some $m\in \mathbb{N}^{+}$ ($n<m$).
Without loss of generality, we assume that the choice of $m$ is minimal and $%
d(a,b)\geq \rho ^{n+1}$ in this case, otherwise $a,b$ belong to a same $(n+1)
$-cell and we can choose $|w^{\prime }|=n+1$ such that $b\in V_{w^{\prime }}$
and $a\in K_{w^{\prime }}\cap V_{m}$. We pick a decreasing sequence of cells 
$\{K_{w_{k}}\}_{k=n}^{m}$ such that $|w_{k}|=k$ with $a\in K_{w_{m}}\cap
V_{m}$, $b\in K_{w_{n}}\cap V_{n}$. Then we obtain a sequence of vertices $%
\{b=x_{n},x_{n+1},\cdots ,x_{m}=a\}$, where $x_{k}\in K_{w_{k}}\cap V_{k}$
for $k=n,\cdots ,m$ and $x_{k+1}\in V_{w_{k}}$ for $k=n,\cdots ,m-1$. By
using the Case $1$, 
\begin{eqnarray*}
|H_{p}(u)(a)-H_{p}(u)(b)| &\leq
&\sum_{k=n}^{m-1}|H_{p}(u)(x_{k})-H_{p}(u)(x_{k+1})| \\
&\leq &C_{2}^{1/p}\sum_{k=n}^{m-1}d(x_{k},x_{k+1})^{\sigma _{p}^{\#}-\alpha
/p}\left( \mathcal{E}_{p,\infty }^{\sigma _{p}^{\#}}(H_{p}(u))\right) ^{1/p}
\\
&\leq &C_{2}^{1/p}\left( \mathcal{E}_{p,\infty }^{\sigma
_{p}^{\#}}(H_{p}(u))\right) ^{1/p}\sum_{k=n}^{m-1}\rho ^{k(\sigma
_{p}^{\#}-\alpha /p)} \\
&\leq &C_{3}\rho ^{n(\sigma _{p}^{\#}-\alpha /p)}\left( \mathcal{E}%
_{p,\infty }^{\sigma _{p}^{\#}}(H_{p}(u))\right) ^{1/p} \\
&\leq &C_{3}\rho ^{-(\sigma _{p}^{\#}-\alpha /p)}d(a,b)^{\sigma
_{p}^{\#}-\alpha /p}\left( \mathcal{E}_{p,\infty }^{\sigma
_{p}^{\#}}(H_{p}(u))\right) ^{1/p} \\
&\leq &C_{4}d(a,b)^{\sigma _{p}^{\#}-\alpha /p}\left( \mathcal{E}%
_{n}^{\sigma _{p}^{\#}}(u)\right) ^{1/p}\text{ \ (by (\ref{En-1}))},
\end{eqnarray*}%
where $C_{4}=C_{3}C^{1/p}\rho ^{-(\sigma _{p}^{\#}-\alpha /p)}$. Therefore,
we prove (\ref{HHP}) in this case.

Case $3$: We cannot find a cell that contains $a,b$ with one of which being
a vertex of this cell. By \cite[Proposition 2.5]{GuLau.2020.AASFM}, \emph{%
Condition (H)} holds for $K$. That is, there exists $c>0$ depending only on $%
K$ such that, $d(x,y)<c\rho ^{m}$ ($m\geq 1$) if and only if $x$ and $y$ lie
in the same or neighboring $m$-cells. Assume that 
\begin{equation*}
c\rho ^{m+1}\leq d(a,b)<c\rho ^{m},
\end{equation*}
then by condition (H), $a$ and $b$ lie in the same or neighboring $m$-cells
since $d(a,b)<c\rho ^{m}$. Therefore, we can find $\xi \in V_{w^{(1)}}\cap
V_{w^{(2)}}$ with $|w^{(1)}|=|w^{(2)}|=m$ ($w^{(1)}=w^{(2)}$ when $a$ and $b$
lie in the same $m$-cell) such that $a,\xi \in K_{w^{(1)}}$ and $b,\xi \in
K_{w^{(2)}}$. Therefore, two pairs $(a,\xi )$, $(b,\xi )$ satisfying the
assumption in Case $2$ and%
\begin{equation*}
\max \{d(a,\xi ),d(b,\xi )\}\leq \rho ^{m}\leq c^{-1}\rho ^{-1}d(a,b).
\end{equation*}

Then by the argument in Case $2$, 
\begin{eqnarray*}
|H_{p}(u)(a)-H_{p}(u)(b)| &\leq &|H_{p}(u)(a)-H_{p}(u)(\xi )|+|H_{p}(u)(\xi
)-H_{p}(u)(b)| \\
&\leq &C_4\left( d(a,\xi )^{\sigma _{p}^{\#}-\alpha /p}+d(\xi ,b)^{\sigma
_{p}^{\#}-\alpha /p}\right) \left( \mathcal{E}_{p,\infty }^{\sigma
_{p}^{\#}}(H_{p}(u))\right) ^{1/p} \\
&\leq &C_5 d(a,b)^{\sigma _{p}^{\#}-\alpha /p}\left( \mathcal{E}_{n}^{\sigma
_{p}^{\#}}(u)\right) ^{1/p}\text{ \ (by (\ref{En-1}))},
\end{eqnarray*}%
where $C_5=2C_4c^{-1}\rho ^{-1}$. Thus showing (\ref{HHP}), the proof is complete.
\end{proof}

\subsection{The proof of (NE)}

We are now in a position to verify (NE) for nested fractals.

\begin{proof}[Proof of Theorem \protect\ref{thm1}]
%Firstly,  $KS_{p,\infty}^{\sigma^_p\#}$ is a dense subspace of $L^p(K,\mu)$.
Let $K$ be a nested fractal and $u\in KS_{p,\infty }^{\sigma _{p}^{\#}}$. We
consider the average value of $u$ on small cells. For any integer $n>\frac{%
\ln(c)}{\ln(\rho)}+1$ where $c$ is from Lemma \ref{Hp}, denote $\hat{u}%
_{n}:V_{n}\rightarrow \mathbb{R}$ by 
\begin{equation}
\hat{u}_{n}(\xi )=\frac{1}{\mu (K_{n+1}^{\ast }(\xi ))}\int_{K_{n+1}^{\ast
}(\xi )}u(x)d\mu (x),  \label{f_n}
\end{equation}%
where $\xi \in V_{n}$ and $K_{m}^{\ast }(\xi )$ stands for the union of all $%
m$-cells containing $\xi$ for positive integer $m$. By the geometry of
nested fractals, for any $x\in K_{n}^{\ast }(\xi )$, we have $K_{n+1}^{\ast
}(\xi )\subset B(x,3\rho ^{n})$.

Let $H_{p}(\hat{u}_{n})$ be the auxiliary extension of $\hat{u}_{n}$
constructed in (\ref{HP}) and we choose it is H\"{o}lder continuous
extension to $K$ (still denote it by $H_{p}(\hat{u}_{n})$). Then by \cite[%
the proof of Lemma 2.4]{GaoYuZhang2022PA} (note that we only need the fact
that $H_{p}(\hat{u}_{n})\in C(K)$ to ensure the weak $*$-convergence of the
measure), 
\begin{align}
\Phi _{H_{p}(\hat{u}_{n})}^{\sigma _{p}^{\#}}(r)& \leq C_{1}\mathcal{E}%
_{p,\infty }^{\sigma _{p}^{\#}}(H_{p}(\hat{u}_{n}))  \notag \\
&\leq C_{2}\mathcal{E}_{n}^{\sigma _{p}^{\#}}(\hat{u}_{n})\text{ \ (by (\ref%
{En-1}))}.  \label{thm1_1}
\end{align}

For any $x,y\in V_{w}$ with some $|w|=n$, by using the H\"{o}lder
inequality, we have 
\begin{align}
|\hat{u}_{n}(x)-\hat{u}_{n}(y)|& \leq \frac{C_{3}}{\mu (K_{n+1}^{\ast
}(x))\mu (K_{n+1}^{\ast }(y))}\int_{K_{n+1}^{\ast }(x)}\int_{K_{n+1}^{\ast
}(y)}|u(z)-u(\xi )|d\mu (\xi )d\mu (z)  \notag \\
& \leq \left( \frac{C_{3}}{\mu (K_{n+1}^{\ast }(x))\mu (K_{n+1}^{\ast }(y))}%
\int_{K_{n+1}^{\ast }(x)}\int_{K_{n+1}^{\ast }(y)}|u(z)-u(\xi )|^{p}d\mu
(\xi )d\mu (z)\right) ^{{1/p}}  \notag \\
& \leq \left( \frac{C_{3}}{\rho ^{2n\alpha }}\int_{K_{n+1}^{\ast
}(x)}\int_{B(z,3\rho ^{n})}|u(z)-u(\xi )|^{p}d\mu (\xi )d\mu (z)\right)
^{1/p},  \label{KSDen-1}
\end{align}%
where we have used the $\alpha $-regularity 
\begin{equation*}
\mu (K_{n+1}^{\ast }(x))\asymp \mu (K_{n+1}^{\ast }(y))\asymp \rho ^{n\alpha
}
\end{equation*}%
and that $K_{n+1}^{\ast }(y)\subset B(z,3\rho ^{n})$ for any $z\in
K_{n+1}^{\ast }(x)$ in the last line, since for any $\tilde{y}\in
K_{n+1}^{\ast }(y)$, 
\begin{equation*}
d(\tilde{y},z)\leq d(\tilde{y},y)+d(x,y)+d(x,z)\leq \rho ^{n+1}+\rho
^{n}+\rho ^{n+1}<3\rho ^{n}.
\end{equation*}

Therefore, we have by (\ref{KSDen-1}) that%
\begin{eqnarray}
\sum_{x,y\in V_{w},|w|=n}|\hat{u}_{n}(x)-\hat{u}_{n}(y)|^{p} &\leq
&\sum_{x,y\in V_{w},|w|=n}\frac{C_{3}}{\rho ^{2n\alpha }}\int_{K_{n+1}^{\ast
}(x)}\int_{B(z,3\rho ^{n})}|u(z)-u(\xi )|^{p}d\mu (\xi )d\mu (z)  \notag \\
&\leq &\frac{C_{4}}{\rho ^{2n\alpha }}\int_{K}\int_{B(z,3\rho
^{n})}|u(z)-u(\xi )|^{p}d\mu (\xi )d\mu (z).  \label{KSDen}
\end{eqnarray}

Thus, for $r\in (0,1)$, we have by \eqref{thm1_1} and (\ref{KSDen}) that 
\begin{align}
\Phi _{H_{p}(\hat{u}_{n})}^{\sigma _{p}^{\#}}(r)\leq & C_{2}\rho
^{-n(p\sigma _{p}^{\#}-\alpha )}\sum_{x,y\in V_{w},|w|=n}|\hat{u}_{n}(x)-%
\hat{u}_{n}(y)|^{p}  \notag \\
\leq & \frac{C_{2}C_4}{\rho ^{n(p\sigma _{p}^{\#}+\alpha )}}%
\int_{K}\int_{B(z,3\rho ^{n})}|u(z)-u(\xi )|^{p}d\mu (\xi )d\mu (z).
\label{thm1_2}
\end{align}

Then we will prove that $H_{p}(\hat{u}_{n})$ converges to $u$ uniformly in $%
K $ when $n\rightarrow \infty $, which will be used to obtain (NE). Note
that for $x\in K_{w}$ ($|w|=n$), 
\begin{equation*}
|H_{p}(\hat{u}_{n})(x)-u(x)|\leq |H_{p}(\hat{u}_{n})(x)-H_{p}(\hat{u}%
_{n})(v)|+|\hat{u}_{n}(v)-u(x)|,
\end{equation*}%
where we pick a vertex $v$ of $K_{w}$ and use the fact that $H_{p}(\hat{u_{n}%
})(v)=\hat{u}_{n}(v)$. By Lemma \ref{Hp}, we have 
\begin{equation}
|H_{p}(\hat{u}_{n})(x)-H_{p}(\hat{u}_{n})(v)|\leq C_{5}\rho ^{n(p\sigma
_{p}^{\#}-\alpha )/p}\left( \mathcal{E}_{n}^{\sigma _{p}^{\#}}(\hat{u}%
_{n})\right) ^{1/p}.  \label{f1}
\end{equation}

We claim that $\mathcal{E}_{p,\infty }^{\sigma _{p}^{\#}}(\hat{u}_{n})$ is
finite. Indeed, we obtain by (\ref{KSDen}) that 
\begin{eqnarray}
\mathcal{E}_{n}^{\sigma _{p}^{\#}}(\hat{u}_{n}) &\leq &\rho ^{-n(p\sigma
_{p}^{\#}-\alpha )}\frac{C_{4}}{\rho ^{2n\alpha }}\int_{K}\int_{B(z,3\rho
^{n})}|u(z)-u(\xi )|^{p}d\mu (\xi )d\mu (z)  \notag \\
&=&C_{4}\rho ^{-n(p\sigma _{p}^{\#}+\alpha )}\int_{K}\int_{B(z,3\rho
^{n})}|u(z)-u(\xi )|^{p}d\mu (\xi )d\mu (z)  \notag \\
&\leq &3^{p\sigma _{p}^{\#}+\alpha }C_{4}[u]_{B_{p,\infty }^{\sigma
_{p}^{\#}}}^{p},  \label{KS}
\end{eqnarray}%
which indicates that $\mathcal{E}_{p,\infty }^{\sigma _{p}^{\#}}(\hat{u}%
_{n})<\infty $. Note that for any $y\in K_{n+1}^{\ast }(v)$, 
\begin{equation*}
d(x,y)<d(x,v)+d(v,y)\leq \rho ^{n}+\rho ^{n+1}<2\rho ^{n},
\end{equation*}%
by \cite[Lemma 2.1]{GaoYuZhang2022PA},%
\begin{equation*}
|u(x)-u(y)|\leq C_6 d(x,y)^{(p\sigma _{p}^{\#}-\alpha )/p}[u]_{B_{p,\infty
}^{\sigma _{p}^{\#}}}\leq C_{7}\rho ^{n(p\sigma _{p}^{\#}-\alpha
)/p}[u]_{B_{p,\infty }^{\sigma _{p}^{\#}}}
\end{equation*}
since $u\in KS_{p,\infty }^{\sigma _{p}^{\#}}=$ $B_{p,\infty }^{\sigma
_{p}^{\#}}$. Therefore,%
\begin{align}
|\hat{u}_{n}(v)-u(x)|& =\left\vert \frac{1}{\mu (K_{n+1}^{\ast }(v))}%
\int_{K_{n+1}^{\ast }(v)}(u(x)-u(y))d\mu (y)\right\vert  \notag \\
& \leq \frac{1}{\mu (K_{n+1}^{\ast }(v))}\int_{K_{n+1}^{\ast
}(v)}|u(x)-u(y)|d\mu (y)  \notag \\
& \leq C_{7}\rho ^{n(p\sigma _{p}^{\#}-\alpha )/p}[u]_{B_{p,\infty }^{\sigma
_{p}^{\#}}}.  \label{f2}
\end{align}

Thus, we see from \eqref{f1} and \eqref{f2}, for any $x\in K_{w}$ with $%
|w|=n $, 
\begin{equation*}
|H_{p}(\hat{u}_{n})(x)-u(x)|\leq C_{8}\rho ^{n(p\sigma _{p}^{\#}-\alpha
)/p}\left( \left( \mathcal{E}_{n}^{\sigma _{p}^{\#}}(\hat{u}_{n})\right)
^{1/p}+[u]_{B_{p,\infty }^{\sigma _{p}^{\#}}}\right) ,
\end{equation*}%
which implies that $H_{p}(\hat{u}_{n})$ converge to $u$ uniformly in $K$ as $%
n\rightarrow \infty $.

Finally, we obtain 
\begin{eqnarray*}
\Phi _{u}^{\sigma _{p}^{\#}}(r) &\leq &\liminf_{n\rightarrow \infty }\Phi
_{H_{p}(\hat{u}_{n})}^{\sigma _{p}^{\#}}(r)\text{ (by using Fatou 's Lemma)}
\\
&\leq &C^{\prime }\liminf_{n\rightarrow \infty }\Phi _{u}^{\sigma
_{p}^{\#}}(3\rho ^{n})\text{ (by taking }\liminf_{n\rightarrow \infty }\text{
in }\eqref{thm1_2}\text{)} \\
&\leq &C^{\prime \prime }\liminf_{r\rightarrow 0}\Phi _{u}^{\sigma
_{p}^{\#}}(r)
\end{eqnarray*}%
for $r\in (0,1)$, which implies 
\begin{equation*}
\sup_{r\in (0,1)}\Phi _{u}^{\sigma _{p}^{\#}}(r)\leq C^{\prime \prime
}\liminf_{r\rightarrow 0}\Phi _{u}^{\sigma _{p}^{\#}}(r),
\end{equation*}%
thus showing (NE). The proof is complete.
\end{proof}

\begin{acknowledgement}
The authors were supported by National Natural Science Foundation of China
(11871296).
\end{acknowledgement}

\end{document}

%% file: tcirm.tex
\def\Xint#1{\mathchoice
{\XXint\displaystyle\textstyle{#1}}%
{\XXint\textstyle\scriptstyle{#1}}%
{\XXint\scriptstyle\scriptscriptstyle{#1}}%
{\XXint\scriptscriptstyle\scriptscriptstyle{#1}}%
\!\int}
\def\XXint#1#2#3{{\setbox0=\hbox{$#1{#2#3}{\int}$ }
\vcenter{\hbox{$#2#3$ }}\kern-.6\wd0}}

\def\oint{\Xint-}

\def\FRAME#1#2#3#4#5#6#7#8
{
 \begin{figure}[ptbh]
 \begin{center}
 \includegraphics[height=#3]{#7}
 \caption{#5}
 \label{#6}
 \end{center}
 \end{figure}
}

\def\enddoc{\end{document}}